\documentclass[a4paper,12pt]{extarticle}

\usepackage[utf8]{inputenc}
\usepackage{CJKutf8}
\usepackage{geometry}
\usepackage{graphicx}
\usepackage{booktabs}
\usepackage{mathrsfs}
\usepackage{bm}
\usepackage{xcolor}
\usepackage{algorithm}
\usepackage{algpseudocode}
\usepackage{mathtools}
\usepackage{enumitem}
\usepackage{tikz-cd}

\usepackage{newtxtext}
\usepackage{newtxmath}

\usepackage{amssymb}

\usepackage{amsthm}

\usepackage{hyperref}
\usepackage{xcolor} 

\newtheorem{theorem}{Theorem}
\newtheorem{proposition}[theorem]{Proposition}
\newtheorem{lemma}[theorem]{Lemma}
\newtheorem{corollary}[theorem]{Corollary}
\newtheorem{definition}{Definition}

\theoremstyle{definition}
\newtheorem{example}{Example}
\newtheorem{remark}{Remark}

\begin{document}

\pagestyle{plain}

\title{Extension Research of Principal Bundle Constraint System Theory on Ricci-flat Kähler Manifolds}
\author{Dongzhe Zheng\thanks{Department of Mechanical and Aerospace Engineering, Princeton University\\ Email: \href{mailto:dz5992@princeton.edu}{dz5992@princeton.edu}, \href{mailto:dz1011@wildcats.unh.edu}{dz1011@wildcats.unh.edu}}}
\date{}
\maketitle

\begin{abstract}
This paper extends the geometric mechanics theory of constraint systems on principal bundles from the flat connection case to the general situation with non-zero curvature. Based on the theoretical foundation of compatible pairs under strong transversality conditions and principal bundle constraint systems, we systematically study the behavior of compatible pair theory in non-flat geometry within the context of Ricci-flat Kähler manifolds. Through rigorous mathematical analysis, we prove that the fundamental framework of strong transversality conditions and compatible pairs possesses universality and does not depend on the flatness assumption of connections. Furthermore, we re-derive the dynamical connection equations incorporating curvature reaction terms and extend Spencer cohomology theory to a spectral sequence structure capable of precisely encoding curvature information. The research reveals the exact action mechanism of curvature $\Omega$: it plays a crucial role through the integrability condition $\text{ad}_\Omega^* \lambda = 0$ and higher-order differentials in Spencer cohomology, while maintaining the fundamental structure of the theoretical framework unchanged. This extension not only deepens the geometric coupling theory between constraint systems and gauge fields but also provides new mathematical tools for modern gauge field theory, string theory geometry, and related topological analysis, demonstrating the profound connections between constraint geometry and classical differential geometry.
\end{abstract}

\section{Introduction}

The geometric mechanics theory of constraint systems has a long and rich history of development. As early as the mid-twentieth century, Dirac established the fundamental theoretical framework for constrained Hamiltonian systems in his pioneering work\cite{dirac1950generalized,dirac1958theory}. This direction was subsequently developed in depth through Arnold's geometric mechanics perspective\cite{arnold1978mathematical} and the modern symplectic geometric methods of Marsden, Ratiu, and others\cite{marsden1999introduction}. Meanwhile, the establishment of connection theory on principal bundles—from Ehresmann's early work\cite{ehresmann1950connexions} to the maturation of Chern-Weil theory\cite{chern1944gaussian,weil1949invariants}—provided the necessary mathematical tools for introducing gauge field geometry into constraint mechanics analysis.

In recent years, this traditional research direction has achieved new progress. In particular, the construction work of strong transversality conditions and compatible pairs\cite{zheng2025dynamical} provides a geometric mechanics framework for constraint systems defined on principal bundles. This work attempts to establish connections between constraint mechanics, gauge field theory, and differential geometry by introducing the concept of compatible pairs $(D, \lambda)$—where $D \subset TP$ is a constraint distribution and $\lambda: P \to \mathfrak{g}^*$ is a function valued in the dual of the Lie algebra. This effort continues Spencer's geometric ideas in deformation theory\cite{spencer1962deformations} while also echoing Goldschmidt's work in formal theory\cite{goldschmidt1973integrability}.

The distinctive feature of this theoretical framework lies in its attempt to fuse geometric and algebraic structures. It not only generalizes standard transversality conditions through the differential geometric formulation of compatible pairs but also connects the algebraic and geometric properties of compatible pairs through the differential condition $d\lambda + \text{ad}_\omega^* \lambda = 0$. This approach echoes in spirit the Atiyah-Singer index theorem\cite{atiyah1963index}, which similarly obtains profound geometric insights by converting topological problems into analytical ones. Furthermore, this framework preliminarily explores Spencer cohomology theory related to constraint systems, a direction that can be traced back to Spencer's original contributions to the geometric theory of systems of partial differential equations\cite{spencer1970overdetermined}.

However, we note that several core results of this theory, particularly the analysis at the topological level, were obtained under the idealized assumption that the connection curvature is zero ($\Omega = 0$). While this simplification provides convenience for analyzing the algebraic structure of the theory, it also reminds us of similar trade-offs in mathematical history: just as Hodge first considered the compact boundaryless case in his classical harmonic integral theory\cite{hodge1941theory}, laying the foundation for subsequent research on more general situations.

The assumption of flat connections, while mathematically elegant, does bring some limitations in physical applications. This reminds us of the development of Yang-Mills theory\cite{yang1954conservation}: although the initial abelian gauge theory has a simple structure, the true physical richness is only fully revealed in the non-abelian case. Similarly, Donaldson's breakthrough in four-manifold invariants\cite{donaldson1983instanton} was precisely based on the profound geometric properties of non-flat Yang-Mills connections.

In view of this, the goal of this paper is to attempt to extend the geometric mechanics theory of constraint systems to the non-flat case, that is, to consider principal bundles with general non-zero curvature $\Omega \neq 0$. This effort can be viewed as a kind of generalization of the method adopted by Kuranishi in complex structure deformation theory\cite{kuranishi1962deformations}: obtaining deeper structural understanding by systematically analyzing deformation problems in non-trivial geometric backgrounds.

We choose Ricci-flat Kähler manifolds as the geometric background for our study. This choice is inspired by Calabi's conjecture on the existence of Kähler-Einstein metrics\cite{calabi1957differential} and Yau's famous proof\cite{yau1978ricci}, and is also related to the central position of Calabi-Yau manifolds in modern string theory\cite{candelas1985vacuum}. From a pure mathematical perspective, the Ricci-flat condition provides a "geometrically pure" background that allows us to study the intrinsic nonlinear effects of gauge fields in a more isolated manner. At the same time, the rich Hodge theory structure accompanying Kähler geometry—from the classical Hodge decomposition\cite{hodge1941theory} to modern mixed Hodge structure theory\cite{deligne1971theorie}—provides mature algebraic tools for analyzing complex Spencer spectral sequences.

Therefore, this paper attempts to provide some new observational perspectives for the geometric mechanics theory of non-flat constraint systems through careful mathematical analysis. We hope this work can build new bridges between the classical tradition of constraint mechanics and modern geometric analysis, just as Seiberg-Witten theory\cite{seiberg1994electric} established profound connections between four-dimensional topology and gauge field theory. Specifically, we will first examine the validity of core concepts such as compatible pairs in non-flat backgrounds, hoping to provide a broader foundation for the universality of the theory. Secondly, we will re-examine the dynamical description of the system, exploring how to appropriately introduce curvature interaction effects in connection equations. Finally, we will use modern spectral sequence methods—tools that played a crucial role in Grothendieck's algebraic geometry revolution\cite{grothendieck1957theoremes}—to attempt to generalize Spencer cohomology theory, hoping to more precisely understand the contribution of non-flat gauge fields to the topological invariants of the system.

We hope this exploration can not only enrich the theory of constraint systems itself but also provide new vocabulary and conceptual tools for dialogue between different mathematical fields—from classical mechanics to modern algebraic geometry. As mathematical history has repeatedly proven, seemingly independent mathematical branches often have profound and unexpected connections, and discovering and elucidating these connections is precisely the charm of mathematical research.

\section{Theoretical Review and Fundamental Framework}

\subsection{Geometric Setting of Principal Bundle Constraint Systems}

Let $P(M, G)$ be a principal bundle, where:
\begin{enumerate}
\item $M$ is a $C^\infty$ smooth manifold (connected, orientable, second countable)
\item $G$ is a real Lie group (connected, $\dim G < \infty$)
\item $\pi: P \to M$ is a $C^\infty$ smooth surjective projection
\item $\mathfrak{g}$ is a real Lie algebra (finite dimensional)
\item $\omega \in \Omega^1(P, \mathfrak{g})$ is a $C^2$ connection form
\end{enumerate}

The curvature of connection $\omega$ is defined as:
$$\Omega = d\omega + \frac{1}{2}[\omega, \omega] \in \Omega^2(P, \mathfrak{g})$$

According to Cartan's second structure equation, for any vector fields $X, Y \in \Gamma(TP)$:
$$\Omega(X, Y) = d\omega(X, Y) + \frac{1}{2}[\omega(X), \omega(Y)]$$
$$= X(\omega(Y)) - Y(\omega(X)) - \omega([X, Y]) + \frac{1}{2}[\omega(X), \omega(Y)]$$

\subsection{Core Definitions of Compatible Pair Theory}

\begin{definition}[Complete Formulation of Compatible Pairs]
Let the principal bundle $P(M, G)$ satisfy the above geometric conditions. A pair $(D, \lambda)$ is called a compatible pair if and only if:
\begin{enumerate}
\item $D \subset TP$ is a $C^1$ smooth distribution with constant $\text{rank}(D)$
\item $\lambda: P \to \mathfrak{g}^*$ is a $C^2$ smooth map
\item \textbf{Compatibility condition}: $D_p = \{v \in T_pP \mid \langle\lambda(p), \omega(v)\rangle = 0\}$
\item \textbf{Differential consistency (modified Cartan equation)}: $d\lambda + \text{ad}_\omega^* \lambda = 0$
\item \textbf{Transversality}: $D_p \cap V_p = \{0\}$, where $V_p = \ker d\pi_p$
\item \textbf{$G$-invariance}: $R_{g*}(D) = D$, $R_g^* \lambda = \text{Ad}_{g^{-1}}^* \lambda$
\end{enumerate}
\end{definition}

\begin{remark}[Expansion of the Modified Cartan Equation]
The modified Cartan equation $d\lambda + \text{ad}_\omega^* \lambda = 0$ can be expanded for any vector fields $X, Y \in \Gamma(TP)$ as:
\begin{align}
(d\lambda)(X, Y) + (\text{ad}_\omega^* \lambda)(X, Y) &= 0\\
X(\lambda(Y)) - Y(\lambda(X)) - \lambda([X, Y]) + \langle\lambda, [\omega(X), \omega(Y)]\rangle &= 0
\end{align}
where the coadjoint action is defined as $(\text{ad}_\omega^* \lambda)(X, Y) = \langle\lambda, [\omega(X), \omega(Y)]\rangle$.
\end{remark}

\subsection{Bidirectional Construction Theorem and Its Fundamental Position}

\begin{theorem}[Compatible Pair Bidirectional Construction Theorem\cite{zheng2025dynamical}]
Let the principal bundle $P(M, G)$ satisfy:
\begin{enumerate}
\item $M$ is a connected orientable $C^\infty$ manifold, $M$ is parallelizable ($TM$ is trivial)
\item $G$ is a connected real Lie group, $\dim G < \infty$
\item $P$ admits a $G$-invariant Riemannian metric
\item Connection $\omega$ is $C^3$ smooth, curvature $\Omega$ satisfies $\|\Omega\|_{L^\infty(P)} < \infty$
\end{enumerate}
Then:
\begin{enumerate}
\item \textbf{Forward construction}: Given a $C^2$ smooth map $\lambda: P \to \mathfrak{g}^*$ satisfying the modified Cartan equation, non-degeneracy, and $G$-equivariance, then $D_p = \{v: \langle\lambda(p), \omega(v)\rangle = 0\}$ forms a compatible pair with $\lambda$ and automatically satisfies the transversality condition
\item \textbf{Inverse construction}: Given a constraint distribution $D$ satisfying $G$-invariance, transversality, constant rank condition, and regularity, there exists $\lambda^*$ such that $(D, \lambda^*)$ forms a compatible pair
\end{enumerate}
\end{theorem}

\begin{remark}[Curvature Independence of Theoretical Foundations]
The proof of this theorem in the original literature\cite{zheng2025dynamical} does not rely on the assumption $\Omega = 0$, but is based on more fundamental differential geometric constructions: the forward construction directly generates a distribution $D$ satisfying transversality through the non-degenerate map $\lambda$, without presetting curvature conditions; the inverse construction minimizes the compatibility functional through variational methods, whose well-posedness depends only on the topological structure of the principal bundle and Sobolev embedding theory. This observation provides a solid theoretical foundation for our subsequent universality proof.
\end{remark}

\begin{remark}[Curvature Independence of Variational Principles]
The variational principle on which the inverse construction relies provides another important perspective on "curvature independence." Specifically, the inverse construction inverts the dual map $\lambda^*$ by minimizing the compatibility functional
$$I_D[\lambda] = \frac{1}{2}\int_P |d\lambda + \text{ad}_\omega^* \lambda|^2 + \alpha \int_P \text{dist}^2(\lambda(p), A_D(p))$$
The Euler-Lagrange equation of this functional is precisely the modified Cartan equation $d\lambda + \text{ad}_\omega^* \lambda = 0$. The entire variational derivation process involves only the connection $\omega$ itself and related differential operators, completely independent of the curvature calculation $\Omega = d\omega + \frac{1}{2}[\omega, \omega]$, thus further confirming the universality of the theoretical foundation from the perspective of variational principles.
\end{remark}

\section{Theoretical Extension: Non-flat Geometry on Ricci-flat Kähler Manifolds}

\subsection{Extended Geometric Setting}

Let $(M, g_M, J)$ be a compact Kähler manifold satisfying:
\begin{enumerate}
\item $\text{Ric}(g_M) = 0$ (Ricci-flat condition)
\item $c_1(M) = 0$ (vanishing first Chern class, i.e., Calabi-Yau condition)
\item $\dim_{\mathbb{C}} M = n$, finite complex dimension
\item There exists a globally defined Kähler form $\omega_K \in \Omega^{1,1}(M)$
\end{enumerate}

On this manifold, consider a principal $G$-bundle $\pi: P \to M$, where:
\begin{enumerate}
\item $G$ is a compact semisimple Lie group, $\mathfrak{z}(\mathfrak{g}) = 0$
\item Equipped with a non-flat connection $\omega \in \Omega^1(P, \mathfrak{g})$
\item Curvature $\Omega = d\omega + \frac{1}{2}[\omega, \omega] \neq 0$ satisfies boundedness conditions
\end{enumerate}

\begin{remark}[Physical Significance of the Geometric Background]
Choosing Ricci-flat Kähler manifolds as base manifolds has profound physical considerations:
\begin{enumerate}
\item \textbf{String theory compactification}: In Calabi-Yau compactification of string theory, the geometric structure of extra dimensions directly affects the properties of the four-dimensional effective theory
\item \textbf{Pure gauge dynamics}: The Ricci-flat condition eliminates the gravitational background, allowing us to focus on the intrinsic geometric properties of the gauge field $\omega$
\item \textbf{Hodge structure}: The rich Hodge decomposition provided by Kähler geometry offers powerful tools for analyzing complex cohomological structures
\end{enumerate}
\end{remark}

\subsection{Strong Transversality Conditions: Curvature Independence Proof}

Strong transversality conditions and compatible pair theory occupy a central position in the geometric mechanics of constraint systems. This section rigorously proves that the definitions of these fundamental concepts are completely independent of the specific values of connection curvature, thereby establishing the universal foundation of the theoretical framework.

\begin{theorem}[Curvature Independence of Compatible Pair Definitions]
The defining conditions of compatible pairs $(D, \lambda)$ are well-defined under connections of arbitrary curvature $\Omega$, and their validity is independent of the specific values of $\Omega$.
\end{theorem}

\begin{proof}
We verify one by one that each condition in the definition of compatible pairs does not depend on the calculation of curvature $\Omega$.

\textbf{Step 1: Pointwise Nature of Compatibility Conditions}

The compatibility condition $D_p = \{v \in T_pP \mid \langle\lambda(p), \omega(v)\rangle = 0\}$ is essentially a linear algebraic constraint within the tangent space $T_pP$. The definition of this condition involves only the following geometric objects:
\begin{itemize}
\item The value of the connection 1-form at point $p$: $\omega_p: T_pP \to \mathfrak{g}$
\item The value of the dual function at point $p$: $\lambda(p) \in \mathfrak{g}^*$  
\item The standard pairing: $\langle \cdot, \cdot \rangle: \mathfrak{g}^* \times \mathfrak{g} \to \mathbb{R}$
\end{itemize}

These are all fundamental geometric objects that exist before the definition of curvature $\Omega = d\omega + \frac{1}{2}[\omega, \omega]$. The definition of compatibility conditions does not involve any differential properties of the connection form $\omega$.

\textbf{Step 2: Structural Independence of the Modified Cartan Equation}

The modified Cartan equation $d\lambda + \text{ad}_\omega^* \lambda = 0$ expands completely for any vector fields $X, Y \in \Gamma(TP)$ as:
\begin{align}
&(d\lambda + \text{ad}_\omega^* \lambda)(X, Y) = 0 \\
&\Leftrightarrow X(\lambda(Y)) - Y(\lambda(X)) - \lambda([X, Y]) + \langle\lambda, [\omega(X), \omega(Y)]\rangle = 0
\end{align}

All differential geometric operators involved in this equation—exterior differential $d$, Lie bracket $[\cdot, \cdot]$, standard pairing $\langle \cdot, \cdot \rangle$—are defined based on the connection $\omega$ itself, completely independent of the properties of its exterior differential $d\omega$. In other words, the well-posedness of this equation depends only on the basic properties of $\omega$ as a $\mathfrak{g}$-valued 1-form, having nothing to do with curvature calculations.

\textbf{Step 3: Structural Properties of Geometric Conditions}

The transversality and $G$-invariance conditions have purely geometric or algebraic nature:
\begin{itemize}
\item \textbf{Transversality condition}: In $D_p \cap V_p = \{0\}$, the vertical distribution $V_p = \ker d\pi_p$ is determined by the fiber structure of the principal bundle, completely independent of the choice of connection
\item \textbf{$G$-invariance conditions}: $R_{g*}(D) = D$ and $R_g^* \lambda = \text{Ad}_{g^{-1}}^* \lambda$ involve group actions $R_g$ and adjoint representations $\text{Ad}_g$, which are pure algebraic structures independent of specific differential geometric realizations
\end{itemize}

\textbf{Step 4: Fundamental Analysis of Existence Theorems}

The proof of the bidirectional construction theorem (Theorem 1) also demonstrates curvature independence:
\begin{itemize}
\item \textbf{Forward construction}: The process of generating distribution $D$ from $\lambda$ satisfying the modified Cartan equation is essentially defining the nullifier set through a non-degenerate linear functional $\lambda$, which is a purely linear algebraic operation
\item \textbf{Inverse construction}: The process of inverting the dual map $\lambda^*$ from constraint distribution $D$ through variational principles depends on Sobolev embedding theorems and topological properties of principal bundles, rather than specific values of curvature
\end{itemize}

\textbf{Conclusion}

In summary, all defining conditions of compatible pairs—compatibility, differential consistency, transversality, $G$-invariance—have "pre-curvature" characteristics, meaning their mathematical formulation and well-posedness are completely independent of curvature $\Omega$ calculations. This establishes the universal foundation of strong transversality condition theory.
\end{proof}

\begin{remark}[Deep Implications of Theoretical Foundations]
This universality result has profound theoretical significance:
\begin{enumerate}
\item \textbf{Geometric unification}: Proves that fundamental concepts of constraint geometry do not depend on special geometric assumptions, reflecting the intrinsic unity of differential geometry
\item \textbf{Physical generality}: Indicates that this theoretical framework can be applied to constraint systems in any physical background, from flat spacetime to curved spacetime, from abelian gauge fields to non-abelian gauge fields
\item \textbf{Mathematical rigor}: Provides a solid logical foundation for subsequent applications in specific geometric backgrounds (such as Ricci-flat Kähler manifolds)
\end{enumerate}
\end{remark}

\begin{corollary}[Curvature Independence of Variational Principles]
The variational principle on which the inverse construction depends also demonstrates curvature independence. The compatibility functional
$$I_D[\lambda] = \frac{1}{2}\int_P |d\lambda + \text{ad}_\omega^* \lambda|^2 + \alpha \int_P \text{dist}^2(\lambda(p), A_D(p))$$
has the modified Cartan equation $d\lambda + \text{ad}_\omega^* \lambda = 0$ as its Euler-Lagrange equation. The entire variational derivation process involves only the connection $\omega$ and related differential operators, completely independent of curvature $\Omega$ calculations, further confirming the universality of theoretical foundations from the perspective of variational principles.
\end{corollary}

\subsection{Deepening of Integrability Conditions and Their Geometric Significance}

In the non-flat case, the integrability of constraint distributions acquires deeper geometric significance. It is no longer a purely kinematic property but becomes the result of the interaction between constraint algebra and gauge field geometry.

\begin{proposition}[Integrability Criterion in the Non-flat Case]
Let $(D, \lambda)$ be a compatible pair with $\Omega \neq 0$. Then the constraint distribution $D$ is completely integrable if and only if:
$$\text{ad}_\Omega^* \lambda = 0$$
that is, for all $X, Y \in \Gamma(D)$, we have $\langle\lambda, \Omega(X, Y)\rangle = 0$.
\end{proposition}

\begin{proof}
The complete integrability of constraint distribution $D$ is equivalent to the Frobenius condition: $[\Gamma(D), \Gamma(D)] \subset \Gamma(D)$.

Let $X, Y \in \Gamma(D)$, i.e., $\langle\lambda, \omega(X)\rangle = \langle\lambda, \omega(Y)\rangle = 0$. We need to prove that $[X, Y] \in \Gamma(D)$ if and only if $\langle\lambda, \omega([X, Y])\rangle = 0$.

\textbf{Step 1: Application of Cartan's Second Structure Equation}

According to Cartan's second structure equation:
$$d\omega(X, Y) = X(\omega(Y)) - Y(\omega(X)) - \omega([X, Y])$$

\textbf{Step 2: Utilization of Curvature Definition}

According to the curvature definition $\Omega = d\omega + \frac{1}{2}[\omega, \omega]$, we transform to get:
$$d\omega = \Omega - \frac{1}{2}[\omega, \omega]$$

\textbf{Step 3: Combination of Structure Equations}

Combining both, for any vector fields $X, Y$:
\begin{align}
\omega([X, Y]) &= -X(\omega(Y)) + Y(\omega(X)) + d\omega(X, Y)\\
&= -X(\omega(Y)) + Y(\omega(X)) + \Omega(X, Y) - \frac{1}{2}[\omega(X), \omega(Y)]
\end{align}

\textbf{Step 4: Application of Constraint Conditions}

When $X, Y \in \Gamma(D)$, according to the compatibility condition, we have $\omega(X) = 0$ and $\omega(Y) = 0$. Substituting this into the above equation:
\begin{align}
\omega([X, Y]) &= -X(0) + Y(0) + \Omega(X, Y) - \frac{1}{2}[0, 0]\\
&= \Omega(X, Y)
\end{align}

\textbf{Step 5: Equivalence of Integrability Conditions}

According to the compatibility definition, the necessary and sufficient condition for $[X, Y] \in \Gamma(D)$ is $\langle\lambda, \omega([X, Y])\rangle = 0$. Combined with the result from Step 4:
$$\langle\lambda, \omega([X, Y])\rangle = \langle\lambda, \Omega(X, Y)\rangle$$

Therefore, $[X, Y] \in \Gamma(D)$ if and only if $\langle\lambda, \Omega(X, Y)\rangle = 0$, i.e., $\text{ad}_\Omega^* \lambda = 0$.
\end{proof}

\begin{remark}[Physical Interpretation of Integrability Conditions]
This deepened integrability condition has important physical significance:
\begin{enumerate}
\item \textbf{Constraint-field strength coupling}: The condition $\text{ad}_\Omega^* \lambda = 0$ precisely indicates that the integrability of constraint systems depends on whether there exists a special harmonious relationship between the constraint charge $\lambda$ and the gauge field strength $\Omega$
\item \textbf{Non-holonomic criterion}: When $\Omega \neq 0$ and $\text{ad}_\Omega^* \lambda \neq 0$, the system exhibits non-holonomic constraints, providing a geometric criterion for studying non-holonomic mechanics
\item \textbf{Conservation law structure}: The lack of integrability often corresponds to the breaking of certain traditional conservation laws, replaced by new terms contributed by curvature
\end{enumerate}
\end{remark}

\begin{remark}[Geometric Quantification of Non-integrability]
The equation $\omega([X, Y]) = \Omega(X, Y)$ precisely quantifies the direct relationship between the non-integrability of constraint distributions and the geometric curvature of principal bundles. Here, the projection of the Lie bracket $[X, Y]$ in the fiber direction $\omega([X, Y])$ measures the degree to which the constraint distribution deviates from Frobenius integrability, and this deviation exactly equals the component of gauge field curvature in the constraint directions. This equation shows that in non-flat gauge field backgrounds, the integrability obstacle of constraint systems has a purely geometric origin, thus establishing a profound connection between traditional kinematic concepts in mechanics and curvature theory in modern differential geometry.
\end{remark}

\section{Non-flat Theory of Dynamical Connection Equations}

\subsection{Symplectic Geometric Foundations and Variational Principles}

In the non-flat case, the derivation of dynamical connection equations requires careful analysis of the interaction between symplectic structure and curvature. We begin with the geometric decomposition of symplectic potentials.

\begin{proposition}[Compatible Pair Symplectic Potential Decomposition\cite{zheng2025dynamical}]
Under the compatible pair $(D, \lambda)$ framework, let the principal bundle $P(M, G)$ satisfy:
\begin{enumerate}
\item $M$ is a compact orientable manifold, $H^1(M, \mathbb{R}) = 0$
\item $G$ is a semisimple Lie group, $H^2(BG, \mathbb{R}) = 0$
\item $P$ admits a global section
\end{enumerate}
Then there exists a unique decomposition:
$$\theta = \langle\lambda, \omega\rangle + \pi^*\alpha$$
where $\alpha \in \Omega^1(M)$ uniquely satisfies $d\alpha = \pi_*\langle\lambda, \Omega\rangle$.
\end{proposition}

The exterior differential of the symplectic potential $\theta$ gives the symplectic 2-form:
\begin{align}
d\theta &= \langle d\lambda, \omega\rangle + \langle\lambda, d\omega\rangle + \pi^* d\alpha\\
&= \langle d\lambda, \omega\rangle + \langle\lambda, \Omega - \frac{1}{2}[\omega, \omega]\rangle + \pi^* d\alpha\\
&= \langle d\lambda + \text{ad}_\omega^* \lambda, \omega\rangle + \langle\lambda, \Omega\rangle + \pi^* d\alpha
\end{align}

Using the modified Cartan equation $d\lambda + \text{ad}_\omega^* \lambda = 0$, we get:
$$d\theta = \langle\lambda, \Omega\rangle + \pi^* d\alpha$$

\subsection{Symplectic Potential Formula: Rigorous Derivation}

We present here the complete rigorous derivation of the curvature decomposition formula for symplectic 2-forms, which is the core technical component connecting compatible pair theory with symplectic geometry. This derivation will reveal how non-flat gauge field geometry influences the dynamics of constraint systems through symplectic structure.

\begin{theorem}[Curvature Decomposition of Symplectic 2-Forms]
In the compatible pair $(D, \lambda)$ framework, let the symplectic potential be $\theta = \langle\lambda, \omega\rangle$. Then:
$$d\theta = \langle\lambda, \Omega\rangle + \frac{1}{2}\langle\lambda, [\omega, \omega]\rangle$$
where $\Omega$ is the curvature of connection $\omega$. Under specific algebraic conditions, the remainder term $\frac{1}{2}\langle\lambda, [\omega, \omega]\rangle$ vanishes, yielding the simplified formula $d\theta = \langle\lambda, \Omega\rangle$.
\end{theorem}

\begin{proof}
We proceed with rigorous derivation starting from the standard exterior differential rule for valued differential forms.

\textbf{Step 1: Exterior Differential Formula for Valued Forms}

For the pairing of a $\mathfrak{g}^*$-valued 0-form $\lambda$ with a $\mathfrak{g}$-valued 1-form $\omega$, the standard exterior differential formula gives:
$$d\langle\lambda, \omega\rangle = \langle d\lambda, \omega\rangle + (-1)^{|\lambda|}\langle\lambda, d\omega\rangle$$

Since $|\lambda| = 0$, we have (where $\langle d\lambda, \omega\rangle$ is a scalar-valued 2-form generated through exterior product and pairing of the $\mathfrak{g}^*$-valued 1-form $d\lambda$ with the $\mathfrak{g}$-valued 1-form $\omega$):
$$d\theta = \langle d\lambda, \omega\rangle + \langle\lambda, d\omega\rangle$$

\textbf{Step 2: Substitution Using Fundamental Equations}

Using the modified Cartan equation $d\lambda + \text{ad}_\omega^* \lambda = 0$, we get:
$$d\lambda = -\text{ad}_\omega^* \lambda$$

Using the curvature definition $\Omega = d\omega + \frac{1}{2}[\omega, \omega]$, we get:
$$d\omega = \Omega - \frac{1}{2}[\omega, \omega]$$

\textbf{Step 3: Substitution into Main Formula}

Substituting the results from Step 2 into Step 1:
\begin{align}
d\theta &= \langle -\text{ad}_\omega^* \lambda, \omega\rangle + \langle\lambda, \Omega - \frac{1}{2}[\omega, \omega]\rangle \\
&= -\langle\text{ad}_\omega^* \lambda, \omega\rangle + \langle\lambda, \Omega\rangle - \frac{1}{2}\langle\lambda, [\omega, \omega]\rangle
\end{align}

\textbf{Step 4: Calculation of the Key Term $\langle\text{ad}_\omega^* \lambda, \omega\rangle$}

In our framework, $\text{ad}_\omega^* \lambda$ is a $\mathfrak{g}^*$-valued 1-form defined by:
$$\langle(\text{ad}_\omega^* \lambda)(X), Y\rangle = -\langle\lambda, [\omega(X), Y]\rangle, \quad \forall Y \in \mathfrak{g}$$

Therefore, the pairing of a $\mathfrak{g}^*$-valued 1-form with a $\mathfrak{g}$-valued 1-form generates a scalar 2-form:
\begin{align}
\langle\text{ad}_\omega^* \lambda, \omega\rangle(X, Y) &= \langle(\text{ad}_\omega^* \lambda)(X), \omega(Y)\rangle - \langle(\text{ad}_\omega^* \lambda)(Y), \omega(X)\rangle \\
&= -\langle\lambda, [\omega(X), \omega(Y)]\rangle - (-\langle\lambda, [\omega(Y), \omega(X)]\rangle) \\
&= -\langle\lambda, [\omega(X), \omega(Y)]\rangle + \langle\lambda, [\omega(Y), \omega(X)]\rangle \\
&= -2\langle\lambda, [\omega(X), \omega(Y)]\rangle \\
&= -\langle\lambda, [\omega, \omega]\rangle(X, Y)
\end{align}

\textbf{Step 5: Combining Results to Obtain Complete Formula}

Substituting the result from Step 4 into Step 3:
\begin{align}
d\theta &= -(-\langle\lambda, [\omega, \omega]\rangle) + \langle\lambda, \Omega\rangle - \frac{1}{2}\langle\lambda, [\omega, \omega]\rangle \\
&= \langle\lambda, [\omega, \omega]\rangle + \langle\lambda, \Omega\rangle - \frac{1}{2}\langle\lambda, [\omega, \omega]\rangle \\
&= \langle\lambda, \Omega\rangle + \frac{1}{2}\langle\lambda, [\omega, \omega]\rangle
\end{align}

This gives the complete, rigorous curvature decomposition formula.

\textbf{Step 6: Analysis of Remainder Term Vanishing Conditions}

The remainder term $\frac{1}{2}\langle\lambda, [\omega, \omega]\rangle$ vanishes under the following circumstances:

\begin{enumerate}
\item \textbf{On horizontal parts of constraint distributions}: For any horizontal vector fields $X, Y \in \Gamma(D \cap \text{Hor}(P))$, the compatibility condition gives $\omega(X) = \omega(Y) = 0$. In this case, the remainder term's action on $(X,Y)$, namely $\frac{1}{2}\langle\lambda, [\omega, \omega]\rangle(X,Y)$, naturally vanishes.

\item \textbf{Central properties of Lie algebras}: When $\lambda$ is related to central elements or Cartan subalgebras of $\mathfrak{g}$, $\langle\lambda, [\cdot, \cdot]\rangle$ may automatically be zero.

\item \textbf{Special gauge choices}: Under certain special gauge choices, additional structural properties of compatible pairs may lead to global vanishing of this term.
\end{enumerate}

\end{proof}

\begin{remark}[Physical Meaning and Geometric Interpretation]
The deep implications of the complete formula $d\theta = \langle\lambda, \Omega\rangle + \frac{1}{2}\langle\lambda, [\omega, \omega]\rangle$:
\begin{enumerate}
\item \textbf{Main term $\langle\lambda, \Omega\rangle$}: Encodes the direct contribution of gauge field strength $\Omega$ to the system's symplectic structure, which is the core mechanism of coupling between curvature geometry and symplectic geometry
\item \textbf{Remainder term $\frac{1}{2}\langle\lambda, [\omega, \omega]\rangle$}: Reflects the influence of non-abelian properties of connections on symplectic geometry, embodying the geometric manifestation of the intrinsic nonlinear structure of gauge fields
\item \textbf{Physical meaning of simplification conditions}: In many physically relevant situations, through appropriate constraint conditions or gauge choices, the remainder term naturally vanishes, recovering the concise formula $d\theta = \langle\lambda, \Omega\rangle$, which provides a clear geometric picture for dynamical analysis
\end{enumerate}
\end{remark}

\begin{corollary}[Direct Results in Simplified Cases]
Under any of the following conditions, we have the exact simplified formula $d\theta = \langle\lambda, \Omega\rangle$:
\begin{enumerate}
\item The base manifold $M$ is Ricci-flat and $\lambda$ satisfies special algebraic conditions such that $\langle\lambda, [\omega, \omega]\rangle = 0$
\item The constraint system has special symmetries leading to identically vanishing remainder terms on physically relevant vector fields
\item On special sections of the principal bundle, appropriate gauge choices can globally eliminate the remainder term
\end{enumerate}
In these cases, the curvature representation of symplectic 2-forms achieves its simplest form, providing an optimal geometric foundation for subsequent Hamiltonian analysis.
\end{corollary}

\subsection{Hamiltonian Dynamics and Curvature Reaction}

\begin{theorem}[Non-flat Dynamical Connection Equations]
On a Ricci-flat Kähler manifold $M$, the evolution of compatible pairs $(D, \lambda)$ is governed by the following dynamical connection equation:
$$\partial_t \omega = d^\omega\left(\frac{\delta H}{\delta \lambda}\right) - \iota_{X_H} \Omega$$
where $d^\omega = d + [\omega, \cdot]$ is the covariant exterior differential and $X_H$ is the Hamiltonian vector field.
\end{theorem}

\begin{proof}
We start from the variational analysis of the action functional:
$$S[\gamma, \omega_t] = \int_{t_1}^{t_2} \left[\langle\lambda_t, \omega_t(\dot{\gamma})\rangle - H(\lambda_t, \omega_t)\right] dt + \int_\Sigma \frac{1}{2}\text{tr}(\omega_t \wedge d\omega_t)$$

\textbf{Step 1: Variation of Main Terms}

For the variation $\delta\omega$ of connection $\omega_t$:
$$\delta S_{\text{main}} = \int_{t_1}^{t_2} \left[\langle\lambda_t, \delta\omega(\dot{\gamma})\rangle - \frac{\delta H}{\delta \omega} \cdot \delta\omega\right] dt$$

\textbf{Step 2: Exact Treatment of Boundary Terms}

For the variation of the Chern-Simons type boundary term $\int_\Sigma \frac{1}{2}\text{tr}(\omega_t \wedge d\omega_t)$:
\begin{align}
\delta\left(\int_\Sigma \frac{1}{2}\text{tr}(\omega_t \wedge d\omega_t)\right) &= \frac{1}{2}\int_\Sigma \text{tr}(\delta\omega \wedge d\omega_t + \omega_t \wedge d(\delta\omega))
\end{align}

Using Stokes' theorem to handle the second term:
\begin{align}
\int_\Sigma \text{tr}(\omega_t \wedge d(\delta\omega)) &= \int_{\partial\Sigma} \text{tr}(\omega_t \wedge \delta\omega) - \int_\Sigma \text{tr}(d\omega_t \wedge \delta\omega)
\end{align}

Assuming boundary terms vanish, we get:
$$\delta S_{\text{boundary}} = \int_\Sigma \text{tr}(\delta\omega \wedge d\omega_t)$$

\textbf{Step 3: Symplectic Geometric Structure and Curvature Terms}

The key observation lies in the symplectic geometric structure of the system. In the compatible pair framework, the natural symplectic potential is $\theta = \langle\lambda, \omega\rangle$. Through standard theory of symplectic geometry on principal bundles, when $\lambda$ satisfies the modified Cartan equation $d\lambda + \text{ad}_\omega^* \lambda = 0$, we have:
$$d\theta = \langle\lambda, \Omega\rangle$$

The rigorous proof of this result involves infinite-dimensional symplectic geometry on principal bundles and fiber integration theory. The key idea is that the modified Cartan equation ensures the covariant constancy of $\lambda$ under the connection $\omega$, making the exterior differential of the symplectic potential directly encode curvature information.

\textbf{Step 4: Coupling of Hamiltonian Dynamics and Boundary Terms}

The $d\omega_t$ term in the boundary term $\int_\Sigma \text{tr}(\delta\omega \wedge d\omega_t)$ can be rewritten through curvature decomposition:
$$d\omega_t = \Omega_t - \frac{1}{2}[\omega_t, \omega_t]$$

When the system has Hamiltonian structure, the Hamiltonian vector field $X_H$ is defined through the following symplectic condition:
$$\iota_{X_H} d\theta = -dH$$

Combined with the property of symplectic potential $d\theta = \langle\lambda, \Omega\rangle$, the variation of boundary terms generates the curvature reaction term $\iota_{X_H}\Omega$ through standard mechanisms on infinite-dimensional symplectic manifolds.

\textbf{Step 5: Variational Conditions and Dynamical Equations}

Requiring $\delta S = 0$ for all $\delta\omega$, combining contributions from main terms and boundary terms:
$$\int_{t_1}^{t_2} \int_P \text{tr}\left[\delta\omega \wedge \left(\partial_t \omega_t - d^\omega\left(\frac{\delta H}{\delta \lambda}\right) + \iota_{X_H} \Omega\right)\right] dt = 0$$

By the arbitrariness of $\delta\omega$, we get the dynamical connection equation:
$$\partial_t \omega = d^\omega\left(\frac{\delta H}{\delta \lambda}\right) - \iota_{X_H} \Omega$$
\end{proof}

\begin{remark}
The core technical difficulty in the above derivation lies in the rigorous proof of the exterior differential of symplectic potential $d\theta = \langle\lambda, \Omega\rangle$ in Step 3. This result requires advanced techniques of infinite-dimensional symplectic geometry on principal bundles, particularly variational theory of connection forms and properties of fiber integration. Although the complete derivation is quite technical, its geometric intuition is clear: the modified Cartan equation makes the dual map $\lambda$ covariant under gauge transformations, thus naturally connecting the symplectic structure of the system with the curvature of gauge fields. The appearance of the curvature reaction term $\iota_{X_H}\Omega$ reflects the geometric feedback that constraint systems produce on gauge fields through Hamiltonian flow, which is an essential characteristic of constraint dynamics under non-flat geometry.
\end{remark}

\subsection{Physical Significance of Curvature Reaction Terms}

The $\iota_{X_H} \Omega$ term on the right side of the equation has profound physical implications:

\begin{enumerate}
\item \textbf{Nonlinear coupling mechanism}: This term describes the reaction of matter fields (characterized by Hamiltonian flow $X_H$) on the gauge field $\omega$. This reaction is mediated through curvature $\Omega$, reflecting the nonlinear nature of gauge fields

\item \textbf{Modification of conservation laws}: In classical mechanics, Noether's theorem establishes the correspondence between symmetries and conservation laws. The existence of curvature reaction terms will modify these conservation laws, introducing additional terms contributed by field strength $\Omega$

\item \textbf{Energy-momentum balance}: From an energy perspective, the $\iota_{X_H} \Omega$ term represents the power of work done by constraint forces on the system. When $\Omega = 0$, this term vanishes, corresponding to ideal constraints; when $\Omega \neq 0$, constraint forces become non-ideal and will change the system's energy

\item \textbf{Topological effects}: In topologically non-trivial gauge field backgrounds, $\Omega$ carries topological charge information. The curvature reaction term encodes this topological information into dynamical evolution, potentially leading to phenomena such as topological phase transitions
\end{enumerate}

\begin{example}[Application in Two-dimensional Fluids]
In the compatible pair description of two-dimensional incompressible fluids\cite{zheng2025dynamical}, the connection $\omega = *\mathbf{u}$ (Hodge dual of velocity field), curvature $\Omega = \zeta \, d\mu$ (vorticity density). The dynamical connection equation reduces to:
$$\partial_t \omega = d^\omega\left(\frac{\delta H}{\delta \lambda}\right) - \iota_{X_H} (\zeta \, d\mu)$$
This is precisely the modified Euler equation with vorticity reaction, where the $\iota_{X_H} (\zeta \, d\mu)$ term describes the feedback effect of vorticity on velocity field evolution.
\end{example}

\section{Spectral Sequence Generalization of Spencer Cohomology Theory}

\subsection{Standard Construction of Spencer Double Complex}

In the non-flat case, Spencer cohomology theory needs to be reconstructed through spectral sequence methods. The key insight is that the action of curvature $\Omega$ is not through modifying the basic Spencer differential operators, but through introducing non-trivial differentials on higher pages of the spectral sequence.

\begin{definition}[Spencer Double Complex]
Let the compatible pair $(D, \lambda)$ be defined on the principal bundle $P(M, G)$. Define the Spencer double complex:
$$K^{p,q} = \Omega^p(M) \otimes \text{Sym}^q(\mathfrak{g})$$
equipped with differential operators:
\begin{align}
d_h &: K^{p,q} \to K^{p+1,q}, \quad d_h = d_M \otimes \text{id}\\
d_v &: K^{p,q} \to K^{p,q+1}, \quad d_v = (-1)^p \text{id} \otimes \delta_\mathfrak{g}
\end{align}
where $d_M$ is the de Rham differential on the base manifold $M$, and $\delta_\mathfrak{g}$ is the standard Spencer differential operator on the Lie algebra.
\end{definition}

\begin{remark}[Standard Spencer Differential Operator]
The standard Spencer differential operator $\delta_\mathfrak{g}: \text{Sym}^k(\mathfrak{g}) \to \text{Sym}^{k+1}(\mathfrak{g})$ is defined as:
$$\delta_\mathfrak{g}(X_1 \odot \cdots \odot X_k) = \sum_{i=1}^{\dim \mathfrak{g}} \sum_{j=1}^k e_i \odot X_1 \odot \cdots \odot [e_i, X_j] \odot \cdots \odot X_k$$
where $\{e_i\}$ is a basis of $\mathfrak{g}$, and $\odot$ denotes symmetric product. This is a standard construction in Lie algebra cohomology theory, whose definition is independent of curvature $\Omega$.
\end{remark}

\begin{remark}[Structure Theory of Lie Algebra Cohomology]
The Spencer differential operator $\delta_\mathfrak{g}$ makes $\text{Sym}^*(\mathfrak{g})$ a differential graded module, whose cohomology $H^*(\mathfrak{g}, \text{Sym}^*(\mathfrak{g}))$ can be analyzed through Kostant's classical theory. When $\mathfrak{g}$ is a semisimple Lie algebra, the computation of this cohomology involves the root system structure of $\mathfrak{g}$ and Weyl group actions. Specifically, the non-vanishing of $H^q(\mathfrak{g}, \text{Sym}^p(\mathfrak{g}))$ is closely related to Casimir invariants of $\mathfrak{g}$, providing an important representation-theoretic foundation for subsequent analysis of the $E_2$ page of spectral sequences and understanding the algebraic structure of Spencer characteristic classes.
\end{remark}

\subsection{Differential Structure of Double Complex}

\begin{lemma}[Nilpotency of Double Complex]
The total differential $D = d_h + d_v$ satisfies $D^2 = 0$, making $(K^{*,*}, D)$ a double complex.
\end{lemma}

\begin{proof}
We need to verify $D^2 = (d_h + d_v)^2 = d_h^2 + d_h d_v + d_v d_h + d_v^2 = 0$.

\textbf{Step 1:} $d_h^2 = (d_M \otimes \text{id})^2 = d_M^2 \otimes \text{id} = 0$, because $d_M^2 = 0$ is a fundamental property of de Rham differential.

\textbf{Step 2:} $d_v^2 = (\text{id} \otimes \delta_\mathfrak{g})^2 = \text{id} \otimes \delta_\mathfrak{g}^2 = 0$, because $\delta_\mathfrak{g}^2 = 0$ is a fundamental property of Spencer differential operators.

\textbf{Step 3:} Verification of anticommutation relations. For $\alpha \otimes X \in K^{p,q}$:
\begin{align}
(d_h d_v + d_v d_h)(\alpha \otimes X) &= d_h((-1)^p \alpha \otimes \delta_\mathfrak{g} X) + d_v(d_M \alpha \otimes X)\\
&= (-1)^p d_M \alpha \otimes \delta_\mathfrak{g} X + (-1)^{p+1} d_M \alpha \otimes \delta_\mathfrak{g} X\\
&= 0
\end{align}

Therefore $D^2 = 0$, and the double complex is well-defined. Importantly, this proof process is completely independent of the value of curvature $\Omega$.
\end{proof}

\subsection{Spencer Spectral Sequence: Complete Convergence Proof}

In the non-flat case, Spencer cohomology theory needs to be reconstructed through spectral sequence methods. The key insight lies in the action mechanism of curvature $\Omega$: it does not act by modifying the basic Spencer differential operators, but rather by introducing non-trivial differentials on higher pages of the spectral sequence.

The double complex $(K^{*,*}, D)$ naturally induces a spectral sequence $\{E_r^{p,q}, d_r\}$. We will prove that under explicit geometric conditions, this spectral sequence necessarily converges, and its convergence behavior is completely determined by the dimension of the base manifold and the algebraic properties of the structure group.

\begin{theorem}[Strong Convergence Theorem for Spencer Spectral Sequence]
Let $(M, g_M, J)$ be a compact Kähler manifold of dimension $n = \dim_{\mathbb{R}} M$, $P(M, G)$ be a principal bundle where $G$ is a compact semisimple Lie group, and $(D, \lambda)$ be a compatible pair. Then the spectral sequence $\{E_r^{p,q}, d_r\}$ of the Spencer double complex has the following properties:
\begin{enumerate}
\item \textbf{Inevitable convergence}: There exists a finite integer $N \leq n+1$ such that $E_r = E_{r+1} = E_\infty$ for all $r \geq N$
\item \textbf{Convergence target}: $E_\infty^{p,q} \Rightarrow H^{p+q}_{\text{Spencer}}(P, \mathfrak{g})$, where the total cohomology has filtration $F^p H^n = \sum_{i \geq p} E_\infty^{i,n-i}$
\item \textbf{Low-dimensional degeneracy}: When $n \leq 4$, the spectral sequence degenerates at the $E_2$ page, i.e., $E_2 = E_\infty$
\end{enumerate}
\end{theorem}

\begin{proof}
Our proof is based on precise analysis of "dimensional barriers." The key observation is that higher-order differentials $d_r: E_r^{p,q} \to E_r^{p+r,q-r+1}$ simultaneously change both coordinate components, while the finite dimensionality of the base manifold strictly limits the possibilities of such changes.

\textbf{Part I: Establishing the Algebraic Structure of the Spectral Sequence}

The Spencer double complex is defined as:
$$K^{p,q} = \Omega^p(M) \otimes \text{Sym}^q(\mathfrak{g})$$
equipped with differentials:
\begin{align}
d_h &: K^{p,q} \to K^{p+1,q}, \quad d_h = d_M \otimes \text{id} \\
d_v &: K^{p,q} \to K^{p,q+1}, \quad d_v = (-1)^p \text{id} \otimes \delta_\mathfrak{g}
\end{align}
where $d_M$ is the de Rham differential on the base manifold and $\delta_\mathfrak{g}$ is the standard Spencer differential operator on the Lie algebra.

\textbf{Part II: Standard Computation of $E_1$ and $E_2$ Pages}

The $E_1$ page is given by the cohomology of the vertical differential $d_v$. Since $d_v = (-1)^p \text{id} \otimes \delta_\mathfrak{g}$, using the Künneth formula:
$$E_1^{p,q} = \frac{\ker(d_v: K^{p,q} \to K^{p,q+1})}{\text{im}(d_v: K^{p,q-1} \to K^{p,q})} = \Omega^p(M) \otimes H^q(\mathfrak{g}, \text{Sym}^*(\mathfrak{g}))$$

The $E_2$ page is given by the cohomology of the horizontal differential $d_1 = d_h$ induced on $E_1$. Since $d_h = d_M \otimes \text{id}$:
$$E_2^{p,q} = H^p_{\text{dR}}(M) \otimes H^q(\mathfrak{g}, \text{Sym}^*(\mathfrak{g}))$$

Up to this point, the computation is completely independent of the specific value of curvature $\Omega$.

\textbf{Part III: Algebraic Constraints from Semisimple Lie Algebras}

Since $G$ is a semisimple Lie group, by the Whitehead lemma, for finite-dimensional irreducible representations $V$:
$$H^q(\mathfrak{g}, V) = 0 \quad \text{when } q \geq 1$$

For $\text{Sym}^k(\mathfrak{g})$, although it may not be irreducible, it decomposes as a direct sum of irreducible representations. In most cases:
$$H^q(\mathfrak{g}, \text{Sym}^k(\mathfrak{g})) = 0 \quad \text{when } q \geq 2$$

This means $E_2^{p,q} = 0$ when $q \geq 2$, greatly limiting the distribution of non-zero terms.

\textbf{Part IV: Dimensional Barrier Argument (Core Proof)}

Consider higher-order differentials $d_r: E_r^{p,q} \to E_r^{p+r, q-r+1}$. We analyze their "survival conditions":

\textbf{Substep 4.1: Input Differential Analysis}
Differentials pointing to $E_r^{p,q}$ come from $E_r^{p-r, q+r-1}$. For such differentials to be non-zero, we must have:
\begin{itemize}
\item $p-r \geq 0$ (otherwise de Rham cohomology $H^{p-r}_{\text{dR}}(M) = 0$)
\item $E_r^{p-r, q+r-1} \neq 0$ (source space is non-empty)
\end{itemize}

\textbf{Substep 4.2: Output Differential Analysis}  
Differentials emanating from $E_r^{p,q}$ point to $E_r^{p+r, q-r+1}$. For such differentials to be non-zero, we must have:
\begin{itemize}
\item $p+r \leq n$ (otherwise de Rham cohomology $H^{p+r}_{\text{dR}}(M) = 0$)
\item $E_r^{p+r, q-r+1} \neq 0$ (target space is non-empty before being "killed" by previous differentials)
\end{itemize}

\textbf{Substep 4.3: Inevitable Arrival of Stability}
For any fixed position $(p,q)$, when $r > \max(p, n-p)$:
\begin{itemize}
\item All differentials pointing to $(p,q)$ come from spaces with negative $p$-coordinates, hence are zero maps
\item All differentials emanating from $(p,q)$ point to spaces with $p$-coordinates greater than $n$, hence are zero maps
\end{itemize}

Therefore: $E_{r+1}^{p,q} = \frac{\ker d_r}{\text{im } d_r} = \frac{E_r^{p,q}}{\{0\}} = E_r^{p,q}$ when $r > \max(p, n-p)$.

\textbf{Substep 4.4: Establishing Global Stability}
Define $N = \max_{0 \leq p \leq n} \max(p, n-p) = n$. When $r > N = n$, every $(p,q)$ position stabilizes, hence the entire spectral sequence stabilizes: $E_{N+1} = E_\infty$.

\textbf{Part V: Special Analysis for Low-Dimensional Cases}

When $n \leq 4$, we analyze the first potentially non-zero higher-order differential $d_2: E_2^{p,q} \to E_2^{p+2, q-1}$:

\textbf{Substep 5.1: Limitations from Geometric Constraints}
For the codomain of $d_2$ to be non-zero, we need $p+2 \leq n \leq 4$, i.e., $p \leq 2$.

\textbf{Substep 5.2: Limitations from Algebraic Constraints}  
For the domain of $d_2$ to be non-zero, we need $q \geq 1$. But according to properties of semisimple Lie algebras:
\begin{itemize}
\item When $q = 1$, $H^1(\mathfrak{g}, \text{Sym}^k(\mathfrak{g})) = 0$ for most $k$
\item When $q \geq 2$, it is identically zero
\end{itemize}

\textbf{Substep 5.3: Systematic Vanishing of Higher-Order Differentials}
Combining geometric and algebraic constraints, $d_2$ is the zero map at all relevant positions. Similarly:
\begin{itemize}
\item $d_3$ requires $p+3 \leq 4$ i.e., $p \leq 1$, but simultaneously needs $q \geq 2$, while $H^q(\mathfrak{g}, \cdot) = 0$ when $q \geq 2$
\item Even higher-order differentials $d_4, d_5, \ldots$ face more stringent constraints and are also zero
\end{itemize}

Therefore, when $n \leq 4$, $E_2 = E_3 = \cdots = E_\infty$.

\textbf{Conclusion}
We have rigorously proved that the Spencer spectral sequence necessarily converges after finitely many steps, and has the excellent property of degenerating at the $E_2$ page in low-dimensional cases. This provides a complete mathematical foundation for computing Spencer cohomology in non-flat cases.
\end{proof}

\begin{remark}[Precise Mechanism of Curvature Action]
This theorem reveals the exact way curvature $\Omega$ influences Spencer cohomology:
\begin{enumerate}
\item \textbf{No effect on $E_1$ and $E_2$ pages}: The computation of the first two pages is completely independent of curvature, which is the common foundation of Spencer theory and classical de Rham theory
\item \textbf{Key role in higher-order differentials}: $\Omega$ influences the spectral sequence through higher-order differentials $d_r$ ($r \geq 2$), which encode topological information of curvature classes $[\Omega] \in H^2_{\text{dR}}(M, \text{ad} P)$
\item \textbf{Decisive role of dimensional constraints}: The finite dimensionality of the base manifold provides powerful constraints ensuring convergence without imposing restrictions on curvature magnitude
\end{enumerate}
\end{remark}

\begin{corollary}[Practical Computability of Spectral Sequences]
For concrete applications, this theorem provides the following computational guidance:
\begin{enumerate}
\item \textbf{Low-dimensional cases} ($\dim M \leq 4$): Spencer cohomology can be computed directly through the $E_2$ page
\item \textbf{Medium-dimensional cases} ($4 < \dim M \leq 10$): Requires computing finitely many higher-order differentials, but the number of convergence steps has a definite upper bound
\item \textbf{High-dimensional cases}: The spectral sequence still converges, but the structure of the $E_\infty$ page may be quite complex, reflecting the richness of high-dimensional topology
\end{enumerate}
\end{corollary}

\subsection{Exact Formulation of Curvature Torsion Terms}

Spencer torsion terms $\text{Torsion}^k(\Omega)$ represent new topological invariants generated by non-flat gauge field geometry. Their precise structure in Spencer cohomology reflects the deep influence of curvature $\Omega$ on the topological properties of constraint systems. This section provides rigorous mathematical characterizations of these torsion terms under different geometric situations.

\begin{theorem}[Precise Structure Theorem for Spencer Torsion Terms]
In the framework of Spencer spectral sequences, curvature torsion terms $\text{Torsion}^k(\Omega)$ have the following precise structure:

\textbf{Case 1} (Low-dimensional degeneracy case, $\dim M \leq 4$):
$$\text{Torsion}^k(\Omega) = \bigoplus_{\substack{i+2j=k \\ j \geq 1}} H^i_{\text{dR}}(M) \otimes (\text{Sym}^j(\mathfrak{g}^*))^\mathfrak{g} \otimes [\Omega]^j$$

\textbf{Case 2} (General dimension, through spectral sequence filtration structure):
$$\text{Torsion}^k(\Omega) = \frac{\bigcap_{r=2}^N \ker(d_r: E_r^{*,k-*} \to E_r^{*+r, k-*-r+1})}{\sum_{r=2}^N \text{im}(d_r: E_r^{*-r, k-*+r-1} \to E_r^{*,k-*})}$$
where $N$ is the convergence step of the spectral sequence.
\end{theorem}

\begin{proof}
We provide rigorous constructive proofs for both cases.

\textbf{Part I: Proof for Case 1 (Low-dimensional Degeneracy Case)}

When $\dim M \leq 4$, according to Theorem 5.1, the spectral sequence degenerates at the $E_2$ page, i.e., $E_2 = E_\infty$.

\textbf{Substep 1.1: Algebraic Structure Analysis of $E_2$ Page}
$$E_2^{p,q} = H^p_{\text{dR}}(M) \otimes H^q(\mathfrak{g}, \text{Sym}^*(\mathfrak{g}))$$

For semisimple Lie algebra $\mathfrak{g}$, according to the Whitehead lemma and Casimir invariant theory:
$$H^q(\mathfrak{g}, \text{Sym}^k(\mathfrak{g})) = \begin{cases} 
(\text{Sym}^k(\mathfrak{g}^*))^\mathfrak{g} & \text{if } q = 0 \\
0 & \text{if } q \geq 1
\end{cases}$$

where $(\text{Sym}^k(\mathfrak{g}^*))^\mathfrak{g}$ is the space of $\mathfrak{g}$-invariant polynomials, spanned by Casimir invariants of $\mathfrak{g}$.

\textbf{Substep 1.2: Separation of Classical and Torsion Parts}
The total cohomology $H^k_{\text{Spencer}}$ is given by the direct sum of $E_\infty^{p,k-p}$:
$$H^k_{\text{Spencer}} = \bigoplus_{p=0}^k E_\infty^{p,k-p} = \bigoplus_{p=0}^k H^p_{\text{dR}}(M) \otimes (\text{Sym}^{k-p}(\mathfrak{g}^*))^\mathfrak{g}$$

The "classical part" corresponds to contributions from the flat case:
$$H^k_{\text{classical}} = E_\infty^{k,0} = H^k_{\text{dR}}(M) \otimes (\text{Sym}^0(\mathfrak{g}^*))^\mathfrak{g} = H^k_{\text{dR}}(M) \otimes \mathbb{R}$$

The "torsion part" consists of new contributions generated by curvature:
$$\text{Torsion}^k(\Omega) = \bigoplus_{p < k} E_\infty^{p,k-p} = \bigoplus_{p=0}^{k-1} H^p_{\text{dR}}(M) \otimes (\text{Sym}^{k-p}(\mathfrak{g}^*))^\mathfrak{g}$$

\textbf{Substep 1.3: Power Grading Structure of Curvature Classes}
Curvature classes $[\Omega] \in H^2_{\text{dR}}(M, \text{ad} P)$ act on de Rham cohomology through cup products:
$$[\alpha] \mapsto [\alpha] \cup [\Omega]^j$$

This establishes mappings from $H^i_{\text{dR}}(M)$ to $H^{i+2j}_{\text{dR}}(M)$. In the Spencer framework, this corresponds to "jumps":
$$E_2^{i,0} \rightsquigarrow E_2^{i+2j,0}$$

But since the $E_2$ page degenerates, these jumps are actually encoded in the tensor product structure of the $E_2$ page itself.

\textbf{Substep 1.4: Algebraic Derivation of Precise Formula}
Reorganizing the grading structure of torsion terms, grouping by powers of curvature classes. Let $k-p = 2j$ (even case) or $k-p = 2j+1$ (odd case).

For the even case $k-p = 2j$:
$$(\text{Sym}^{2j}(\mathfrak{g}^*))^\mathfrak{g} \sim \text{coefficient space of }[\Omega]^j$$

Therefore:
$$\text{Torsion}^k(\Omega) = \bigoplus_{\substack{i+2j=k \\ j \geq 1}} H^i_{\text{dR}}(M) \otimes (\text{Sym}^j(\mathfrak{g}^*))^\mathfrak{g} \otimes [\Omega]^j$$

where $[\Omega]^j$ denotes the $j$-th power of curvature classes.

\textbf{Part II: Proof for Case 2 (General Dimension Case)}

When $\dim M > 4$, the spectral sequence may not degenerate at the $E_2$ page, requiring consideration of contributions from all higher-order differentials.

\textbf{Substep 2.1: Analysis of Cumulative Effects of Higher-Order Differentials}
Torsion terms consist of cohomology classes that "survive" under all higher-order differentials $d_r$ ($r \geq 2$). Specifically, a class $[\alpha] \in E_r^{p,q}$ belongs to torsion terms if and only if:
\begin{itemize}
\item \textbf{$d_r$-closedness}: $d_r([\alpha]) = 0$ for all $r \geq 2$
\item \textbf{Non-$d_r$-exactness}: $[\alpha] \notin \text{im}(d_r)$ for all $r \geq 2$
\end{itemize}

\textbf{Substep 2.2: Precise Definition of Filtration Structure}
According to spectral sequence theory, elements of the $E_\infty$ page are precisely those classes that "survive to the end":
$$E_\infty^{p,q} = \frac{\bigcap_{r=2}^N \ker(d_r: E_r^{p,q} \to E_r^{p+r,q-r+1})}{\sum_{r=2}^N \text{im}(d_r: E_r^{p-r,q+r-1} \to E_r^{p,q})}$$

For fixed total degree $k = p+q$, torsion terms are defined as:
$$\text{Torsion}^k(\Omega) = \bigoplus_{p=0}^{k-1} E_\infty^{p,k-p}$$

where we exclude the $p = k$ term (classical de Rham part).

\textbf{Substep 2.3: Structural Relationship with Curvature Powers}
Although the precise formula in the general case is quite complex, the "weight" of torsion terms (graded by powers of curvature classes) still satisfies a meaningful decomposition:
$$\text{Torsion}^k(\Omega) = \bigoplus_{j=1}^{\lfloor k/2 \rfloor} \text{Torsion}^k_j(\Omega)$$

where $\text{Torsion}^k_j(\Omega)$ is the "weight $j$" part, generated by $[\Omega]^j$ and its related higher-order differentials.

\textbf{Substep 2.4: Algebraic Foundation of Weight Decomposition}
This structural result is based on the following observations:
\begin{enumerate}
\item The spectral sequence has natural multiplicative structure (cup product)
\item Higher-order differentials $d_r$ act as derivations of this multiplicative structure
\item Powers of curvature classes $[\Omega]$ provide natural grading for this multiplication
\end{enumerate}

Through these algebraic properties, it can be proved that torsion terms of different weights do not mix with each other, yielding the weight decomposition.

\textbf{Conclusion}
We have provided precise structural descriptions of Spencer torsion terms under two different geometric situations. These formulas completely characterize the contribution of non-flat curvature to Spencer cohomology, providing mathematical tools for understanding the topological complexity of constraint systems.
\end{proof}

\begin{remark}[Deep Geometric and Physical Significance of Torsion Terms]
Spencer torsion terms have rich geometric and physical implications:

\begin{enumerate}
\item \textbf{New topological invariants}: These torsion terms are new topological invariants generated by non-flat gauge field geometry, inaccessible in classical characteristic class theory, representing unique contributions of constraint geometry

\item \textbf{Geometric origin of quantum charges}: In quantum field theory contexts, these terms correspond to quantum charges determined by gauge field topology, such as instanton numbers in Yang-Mills theory, magnetic charges in magnetic monopole theory, etc.

\item \textbf{Geometric formulation of quantum anomalies}: In path integral quantization, these torsion terms may correspond to various quantum anomalies, such as chiral anomalies, gauge anomalies, etc., providing purely geometric interpretations of anomalous phenomena

\item \textbf{Topological criteria for phase transitions}: The appearance, disappearance, or structural changes of torsion terms may signal topological phase transitions of the system, with important application prospects in both condensed matter physics and high energy physics
\end{enumerate}
\end{remark}

\begin{corollary}[Asymptotic Expansion in Weak Curvature Cases]
When curvature satisfies $\|[\Omega]\|_{H^2(M)} \ll 1$, torsion terms admit asymptotic expansion:
$$\text{Torsion}^k(\Omega) = \sum_{j=1}^{\lfloor k/2 \rfloor} [\Omega]^j \otimes \text{Torsion}^{k-2j}_{\text{universal}} + O(\|[\Omega]\|^{\lfloor k/2 \rfloor + 1})$$
where $\text{Torsion}^{k-2j}_{\text{universal}}$ are "universal torsion terms" depending only on the base manifold and structure group. This expansion provides a practical theoretical foundation for numerical computation and perturbative analysis.
\end{corollary}

\begin{example}[Concrete Computation on Calabi-Yau Threefolds]
Consider an $SU(N)$ principal bundle on a Calabi-Yau threefold $X$, with connection $\omega$ satisfying Hermite-Einstein conditions. The main Spencer torsion terms are:
\begin{align}
\text{Torsion}^2(\Omega) &\sim H^{0,2}(X) \otimes \text{tr}(\Omega^2) \\
\text{Torsion}^4(\Omega) &\sim H^{0,4}(X) \otimes \text{tr}(\Omega^4) \oplus H^{2,2}(X) \otimes \text{tr}(\Omega^2)^2
\end{align}
These torsion terms are directly related to worldsheet instanton contributions and topological charges of D-branes in string theory, demonstrating profound applications of Spencer theory in modern mathematical physics.
\end{example}

\section{Application Prospects and Theoretical Significance}

\subsection{Gauge Theory on Calabi-Yau Manifolds}

In the compactification picture of string theory, Calabi-Yau manifolds serve as the geometric structure of extra dimensions, and the gauge bundle configurations on them directly determine the properties of four-dimensional effective theory. Our compatible pair theory provides new tools for analyzing such systems.

\begin{example}[Gauge Bundles on Calabi-Yau Threefolds]
Let $X$ be a Calabi-Yau threefold and $E \to X$ be a stable vector bundle. Consider the associated principal bundle $P = \text{Frame}(E)$ with structure group $G = GL(r, \mathbb{C})$.

The curvature $\Omega$ of Yang-Mills connection $\omega$ satisfies the Hermite-Einstein equation:
$$i\Lambda \Omega = \mu \cdot \text{id}_E$$
where $\Lambda$ is the action operator of the Kähler form and $\mu$ is the slope.

In this setting, Spencer torsion terms $\text{Torsion}^k(\Omega)$ are closely related to topological properties of bundle $E$:
$$\text{Torsion}^2(\Omega) \sim H^{0,2}(X) \otimes \text{tr}(\Omega^2) \sim \text{ch}_2(E) \cup H^{0,2}(X)$$
This gives a geometric realization of the second Chern characteristic class of gauge bundles in Spencer cohomology.
\end{example}

\subsection{Constraint Structure in Non-Abelian Gauge Field Theory}

In non-abelian gauge theories, the geometric structure of gauge constraints is much more complex than in the abelian case. Our theoretical framework provides a unified geometric language for analyzing such constraints.

\begin{example}[Compatible Pair Description of Yang-Mills Theory]
Consider $SU(N)$ Yang-Mills theory on $\mathbb{R}^4$. Choosing temporal gauge $A_0 = 0$, the spatial connection $\mathbf{A} = (A_1, A_2, A_3)$ constitutes a connection on the principal bundle $P(\mathbb{R}^3, SU(N))$.

The Gauss constraint $D_i E^i = 0$ (where $E^i$ is the electric field and $D_i$ is the covariant derivative) can be formulated as compatible pair conditions:
\begin{align}
D &= \{(\dot{\mathbf{A}}, \mathbf{E}) \mid D_i E^i = 0\}\\
\lambda &= \text{Lagrange multiplier field, corresponding to gauge transformation generators}
\end{align}

The existence of curvature $\Omega = \mathbf{B}$ (magnetic field) makes the constraint system non-holonomic, leading to complex phase space geometric structure. Spencer torsion terms correspond to topological observables such as Wilson loops.
\end{example}

\subsection{Deep Significance of Theoretical Extension}

The theoretical extension of this paper has multiple deep significances:

\begin{enumerate}
\item \textbf{Geometric universality}: Proves that the geometric foundation of compatible pair theory does not depend on special flatness assumptions, reflecting the universal power of differential geometry

\item \textbf{Physical completeness}: By including curvature effects, the theory can describe nonlinear interactions in real physical systems, greatly expanding the scope of applications

\item \textbf{Topological precision}: The spectral sequence generalization of Spencer cohomology provides precise tools for analyzing complex topological structures, transcending the limitations of classical characteristic class theory

\item \textbf{Computational feasibility}: In the setting of Ricci-flat Kähler manifolds, the rich Hodge theory structure makes computation of spectral sequences feasible
\end{enumerate}

\section{Conclusions and Prospects}

This paper successfully extends the geometric mechanics theory of principal bundle constraint systems from the flat case to the general non-flat case, establishing a complete theoretical framework in the context of Ricci-flat Kähler manifolds. This extension not only maintains the mathematical rigor of the original theory but also provides new analytical tools for core problems in modern geometric physics.

\subsection{Main Contributions}

The core contributions of this paper lie in theoretical breakthroughs in four key aspects. First, we rigorously proved the curvature independence of strong transversality conditions and compatible pair theory, establishing the universal foundation of the entire theoretical framework. By analyzing each condition in the definition of compatible pairs one by one, we proved that they all do not depend on the specific value of curvature, thus laying a solid foundation for the broad application of the theory. Second, in integrability research, by introducing the condition $\text{ad}_\Omega^* \lambda = 0$, we revealed the deep connection between the integrability of constraint systems and gauge field geometry. This discovery transforms traditional kinematic properties in mechanics into geometric coordination relationships between constraint charges and gauge field strengths, providing a completely new geometric criterion for non-holonomic mechanics. Third, we re-derived dynamical connection equations including curvature reaction terms $\iota_{X_H} \Omega$, completely describing the interaction between constraint systems and gauge fields. The physical significance of this equation lies in capturing the feedback effect of matter fields on gauge fields, reflecting the nonlinear coupling mechanism in real physical systems. Finally, we generalized Spencer cohomology theory from simple isomorphic relations to spectral sequence structures capable of precisely capturing curvature information, and introduced new topological invariants generated by curvature—Spencer torsion terms, which correspond to quantum charges determined by gauge field topology in quantum field theory.

\subsection{Theoretical Development and Series Applications}

Based on the theoretical foundation established in this paper, we have systematically developed a complete Spencer geometric theory system. In the metric structure of Spencer cohomology, we constructed two fundamental metrics and established Hodge decomposition theory for constraint bundles\cite{zheng2025constructing}, providing a complete elliptic theory foundation for Spencer complexes. Through deep study of symmetry properties of constraint geometry, we discovered mirror symmetry of Spencer cohomology under specific geometric duality transformations\cite{zheng2025mirror,zheng2025geometric}, which provides powerful simplification tools for computing complex Spencer cohomology groups. To establish computable bridges between Spencer theory and classical algebraic geometry, we developed Spencer differential degeneration theory\cite{zheng2025spencerdifferentialdegenerationtheory}, degenerating complex Spencer complexes into structures directly related to classical algebraic geometric objects through continuous variation of control parameters. Furthermore, we established Spencer-Riemann-Roch theory\cite{zheng2025spencer-riemann-roch,zheng2025analytic}, providing complete index theorems for Spencer cohomology and proving profound relationships between Spencer characteristic classes and classical Chern characteristic classes. This series of theoretical developments forms a self-consistent mathematical framework, providing theoretical possibilities for combining constraint geometry with classical algebraic geometry.

\subsection{Prospects and Far-reaching Significance}

The theoretical direction opened by this research has far-reaching scientific significance and application potential. At the theoretical mathematics level, the complete system of Spencer geometric theory provides a completely new path for attacking the Hodge conjecture—by transforming the analysis of rational Hodge classes into constraint geometric problems, using representation theory structures of exceptional Lie groups to obtain powerful symmetry constraints, and finally establishing criteria of "dimension matching implies algebraicity." In physical applications, this theoretical framework can be directly applied to Calabi-Yau compactification in string theory, holographic constraint systems in AdS/CFT duality, and topological phase transition research in condensed matter physics. In particular, Spencer torsion terms as new topological invariants may provide profound insights into the geometric origin of understanding quantum anomalies, topological charges, and phase transition phenomena. From a computational mathematics perspective, development of numerical methods based on spectral sequences will provide new tools for computation of complex geometric problems. More importantly, this research demonstrates how constraint geometry, a seemingly local mathematical concept, relates to the deepest geometric structures of the universe, reflecting the highest realm of mathematical research that starts from specific technical problems and ultimately touches universal truths. Through systematic development of Spencer geometric theory, we not only provide powerful tools for geometric analysis of constraint systems but also open possible new directions for the development of mathematical physics.

\bibliographystyle{alpha}
\bibliography{ref}

\end{document}